\documentclass[12pt]{article}
\usepackage{amssymb}
\usepackage{amsthm}

\newtheorem{proposition}{Proposition}[section]
\newtheorem{lemma}{Lemma}[section]

\newtheorem{remark}{Remark}[section]
\newtheorem{example}{Example}[section]

\DeclareSymbolFont{AMSb}{U}{msb}{M}{n}
\DeclareMathSymbol{\Eout}{\mathbin}{AMSb}{"45}
\DeclareMathSymbol{\N}{\mathbin}{AMSb}{"4E}
\DeclareMathSymbol{\PR}{\mathbin}{AMSb}{"50}
\DeclareMathSymbol{\R}{\mathbin}{AMSb}{"52}
\DeclareMathSymbol{\SI}{\mathbin}{AMSb}{"53}
\DeclareMathSymbol{\bU}{\mathbin}{AMSb}{"55}
\DeclareMathSymbol{\bV}{\mathbin}{AMSb}{"56}
\DeclareMathSymbol{\bW}{\mathbin}{AMSb}{"57}
\DeclareMathSymbol{\bX}{\mathbin}{AMSb}{"58}
\DeclareMathSymbol{\bY}{\mathbin}{AMSb}{"59}
\DeclareMathSymbol{\bZ}{\mathbin}{AMSb}{"5A}

\title{Information Geometric Nonlinear Filtering
   \thanks{Preprint of an article submitted for consideration in \textit{Infinite
   Dimensional Analysis, Quantum Probability and Related Topics} \copyright\ 2015,
   copyright World Scientific Publishing Company,
   http://www.worldscientific.com/worldscinet/idaqp}}

\author{Nigel J.~Newton
   \thanks{School of Computer Science and Electronic Engineering, University of Essex,
   Wivenhoe Park, Colchester, CO4 3SQ, UK. ({\tt njn@essex.ac.uk})}}

\begin{document}
\maketitle

\begin{abstract}
This paper develops information geometric representations for nonlinear filters in
continuous time.  The posterior distribution associated with an abstract nonlinear
filtering problem is shown to satisfy a stochastic differential equation on a Hilbert
information manifold. This supports the Fisher metric as a pseudo-Riemannian metric.
Flows of Shannon information are shown to be connected with the quadratic variation of
the process of posterior distributions in this metric.  Apart from providing a suitable
setting in which to study such information-theoretic properties, the Hilbert manifold
has an appropriate topology from the point of view of multi-objective filter
approximations.  A general class of finite-dimensional exponential filters is shown to
fit within this framework, and an intrinsic evolution equation, involving Amari's
$-1$-covariant derivative, is developed for such filters.  Three example systems, one
of infinite dimension, are developed in detail.


\end{abstract}

\newcommand{\ca}{{\cal A}}
\newcommand{\cb}{{\cal B}}
\newcommand{\cd}{{\cal D}}
\newcommand{\cf}{{\cal F}}
\newcommand{\cl}{{\cal L}}
\newcommand{\cm}{{\cal M}}
\newcommand{\cp}{{\cal P}}
\newcommand{\cq}{{\cal Q}}
\newcommand{\cs}{{\cal S}}
\newcommand{\cu}{{\cal U}}
\newcommand{\cv}{{\cal V}}
\newcommand{\cw}{{\cal W}}
\newcommand{\cx}{{\cal X}}
\newcommand{\cy}{{\cal Y}}
\newcommand{\cz}{{\cal Z}}

\newcommand{\relpq}{\cd(P\,|\,Q)}
\newcommand{\relqp}{\cd(Q\,|\,P)}

\newcommand{\Pdot}{{\dot{P}}}
\newcommand{\Qdot}{{\dot{Q}}}

\newcommand{\hhbar}{{\bar{h}}}
\newcommand{\Xbar}{{\bar{X}}}

\newcommand{\atil}{\tilde{a}}
\newcommand{\Btil}{\tilde{B}}
\newcommand{\etil}{\tilde{e}}
\newcommand{\htil}{\tilde{h}}
\newcommand{\Ntil}{\tilde{N}}
\newcommand{\Wtil}{\tilde{W}}
\newcommand{\Xtil}{\tilde{X}}
\newcommand{\xtil}{\tilde{x}}
\newcommand{\ptil}{\tilde{p}}
\newcommand{\util}{\tilde{u}}
\newcommand{\vtil}{\tilde{v}}
\newcommand{\ytil}{\tilde{y}}
\newcommand{\pitil}{\tilde{\pi}}
\newcommand{\etatil}{\tilde{\eta}}
\newcommand{\Omegatil}{\tilde{\Omega}}
\newcommand{\omegatil}{\tilde{\omega}}
\newcommand{\xitil}{\tilde{\xi}}
\newcommand{\zetatil}{\tilde{\zeta}}
\newcommand{\cftil}{\tilde{\cf}}
\newcommand{\PRtil}{\tilde{\PR}}
\newcommand{\Eouttil}{\tilde{\Eout}}

\newcommand{\bfa}{{\bf a}}
\newcommand{\bfe}{{\bf e}}
\newcommand{\bfu}{{\bf u}}
\newcommand{\bfU}{{\bf U}}
\newcommand{\bfv}{{\bf v}}
\newcommand{\bfV}{{\bf V}}
\newcommand{\bfw}{{\bf w}}
\newcommand{\bfW}{{\bf W}}
\newcommand{\bfy}{{\bf y}}
\newcommand{\bfz}{{\bf z}}
\newcommand{\bfZ}{{\bf Z}}

\newcommand{\phat}{\hat{p}}
\newcommand{\Phat}{\hat{P}}

\newcommand{\E}{{\hbox{\bf E}}}
\newcommand{\Emu}{{\hbox{\bf E}_\mu}}
\newcommand{\EP}{{\hbox{\bf E}_P}}
\newcommand{\spanof}{{\hbox{\rm span}}}
\newcommand{\tr}{{\hbox{\rm tr}}}
\newcommand{\half}{{\frac{1}{2}}}
\newcommand{\fndot}{{\,\cdot\,}}
\newcommand{\indic}{{\hbox{\bf 1}}}
\newcommand{\cond}{{\,|\,}}

\section{Introduction} \label{se:intro}

Let $(X_t\in\bX,\,t\ge 0)$ be a Markov ``signal'' process taking values in a
metric space $\bX$, and let $(Y_t\in\R^d,\,t\ge 0)$ be an ``observation'' process
defined by
\begin{equation}
Y_t = \int_0^t h_s(X_s)\,ds + B_t, \label{eq:obsproc}
\end{equation}
where $h:[0,\infty)\times\bX\rightarrow\R^d$ is a Borel measurable function, and
$(B_t\in\R^d,\,t\ge 0)$ is a $d$-vector Brownian motion, independent of $X$.  In this
context, the problem of nonlinear filtering is that of estimating $X_t$, at each time
$t$, from the observations available up to that time, $(Y_s,\,s\in[0,t])$. In order to
compute various optimal estimates of $X_t$, such as the maximum a-posteriori probability
estimate (if $\bX$ is discrete) or the minimum mean-square-error estimate (if $\bX$ is
a normed linear space), it is usually necessary to find, or at least to approximate,
the entire observation-conditional distribution of $X_t$. That a {\em regular} version
of such a distribution exists, and can be represented by an abstract version of Bayes'
formula, is one of the important early developments in the subject \cite{kast1}.
However, starting with the work of Wonham \cite{wonh1} and Shiryayev \cite{shir1},
much of the theory of nonlinear filtering concerns {\em recursive} filtering equations,
in which representations of the posterior distribution are shown to satisfy particular
stochastic differential equations.  The reader is referred to \cite{crro1} for
a wide range of articles on the theory and current practice of the subject.

Recursive filtering equations are typically expressed in ways that are specific to the
nature of the signal space $\bX$.  If, for example, $\bX$ is discrete, then the filter
can be expressed as a stochastic ordinary differential equation for the vector of
posterior probabilities of the individual states, $x\in\bX$ \cite{wonh1,shir1}; whereas,
if $X$ is a multidimensional diffusion process, then the filter can be expressed as a
stochastic partial differential equation for the posterior density \cite{kush1}. One of
the aims of this paper is to unify such results through the use of a filter ``state
space'' that is based on estimation theoretic constructs rather than the underlying
topology of $\bX$.  The state space used is a Hilbert manifold of probability measures
on $\bX$.  This has an appropriate topology for the study of both approximation errors
and information theoretic properties.  These notions are discussed next in the context
of an abstract Bayesian problem, in which the estimand $U:\Omega\rightarrow\bU$ and
observation $V:\Omega\rightarrow\bV$ are defined on a common probability space
$(\Omega,\cf,\PR)$, and take values in measure spaces $(\bU,\cu,\lambda_U)$ and
$(\bV,\cv,\lambda_V)$, respectively.  We assume that $P_{UV}\ll P_U\otimes\lambda_V$,
where $P_{UV}$ is the joint distribution of $(U,V)$, and $P_U$ is the marginal of $U$.
Let $\cp(\cu)$ be the set of probability measures on $\cu$, let
$R_V:\bU\times\bV\rightarrow[0,\infty)$ be a measurable function, for which
$dP_{UV}=R_Vd(P_U\otimes\lambda_V)$, and let
$\Omega^\prime:=\{\omega\in\Omega: 0 < \int_\bU R_V(u,V(\omega))P_U(du)< \infty\}$.
Then $\Omega^\prime\in\cf$, $\PR(\Omega^\prime)=1$, and
$P_{U|V}:\Omega\rightarrow\cp(\cu)$, defined by
\begin{equation}
P_{U|V}(A)
  = \indic_{\Omega^\prime}\frac{\int_A R_V(u,V)P_U(du)}{\int_\bU R_V(u,V)P_U(du)}
    + \indic_{\Omega\setminus\Omega^\prime} P_U(A),  \label{eq:absbay}
\end{equation}
is a {\em regular $V$-conditional distribution} for $U$. (See \cite{mine1} for
details.)

In many applications of Bayesian estimation, including nonlinear filtering, it is not
possible to express $P_{U|V}$ in terms of a finite number of statistics, and so it is
useful to construct approximations: $\Phat:\Omega\rightarrow\cq\subset\cp(\cu)$, where
$\Phat(A)$ is $V$-measurable for all $A$, and $\cq$ is of finite dimension.  Single
estimation objectives, such as minimum mean-square error in the estimate of a
real-valued quantity $f(U)$, induce their own specific measures of approximation error
on $\cp(\cu)$.  On the other hand, if $f$ is sufficiently regular, a more generic
measure of error such as the $L^2$ metric on densities may be useful.  If $\lambda_U$
is a {\em probability} measure, and $P_{U|V}(\omega)$ and $\Phat(\omega)$ have
densities $p_{U|V}(\omega)$ and $\phat(\omega)$ with respect to $\lambda_U$, then the
difference between the minimum mean-square error estimate of $f(U)$ and the mean of $f$
under $\Phat(\omega)$ can be bounded by means of the Cauchy-Schwartz inequality:
\begin{equation}
\left(\E_{P_{U|V}(\omega)}f-\E_{\Phat(\omega)}f\right)^2
   \le \E_{\lambda_U} f^2 \E_{\lambda_U}(p_{U|V}(\omega)-\phat(\omega))^2.
       \label{eq:csbnd}
\end{equation}
Although, in this context, the $L^2$ metric on densities bounds the estimation error,
it may still be poor in practice.  This is so, for example, if $f$ is the indicator
function of a rare, but important, event.  Moreover, we often need generic measures
of error that are suitable for a {\em variety} of objectives.  This is especially
important if the underlying estimation problem is inherently multi-objective.

Multi-objective measures of approximation error are discussed in \cite{newt5}.
One such measure is the Kullback-Leibler (KL) divergence (or ``relative entropy''):
\begin{equation}
\cd(Q\cond P)
  : = \left\{\begin{array}{l}
        \EP\frac{dQ}{dP}\log\frac{dQ}{dP} \quad{\rm if\ }Q\ll P \\
        +\infty \quad{\rm otherwise}.
      \end{array}\right.  \label{eq:kldiv}
\end{equation}
This is widely used in {\em variational} Bayesian estimation. (See, for example,
\cite{smqu1}.)  Apart from its use as a measure of approximation error, the
KL-divergence plays a central role in Shannon information theory.  The
{\em mutual information} between $U$ and $V$ is defined as follows \cite{coth1}:
\begin{equation}
I(U;V) := \cd(P_{UV}\cond P_U\otimes P_V) = \Eout\cd(P_{U|V}\cond P_U).
          \label{eq:mutinf}
\end{equation}
The term $\cd(P_{U|V}|P_U)$, here, can be interpreted as the {\em information gain} of
the posterior distribution $P_{U|V}$ over the prior $P_U$.

Suppose that $W:\Omega\rightarrow\bW$ is a second observation taking values in a
measure space $(\bW,\cw,\lambda_W)$ such that $V$ and $W$ are $U$-conditionally
independent and $P_{UW}\ll P_U\otimes\lambda_W$.  Let
$R_W:\bU\times\bW\rightarrow[0,\infty)$ be a measurable function for which
$dP_{UW} = R_W d(P_U\otimes\lambda_W)$, and let
$\Omega^{\prime\prime}:=\{\omega\in\Omega: 0 < \int_\bU R_W(u,W(\omega))
P_{U|V}(\omega)(du) < \infty\}.$  Then $\Omega^{\prime\prime}\in\cf$,
$\PR(\Omega^{\prime\prime})=1$, and $P_{U|VW}:\Omega\rightarrow\cp(\cu)$, defined by
\begin{equation}
P_{U|VW}(A)
  = \indic_{\Omega^{\prime\prime}}
    \frac{\int_A R_W(u,W)P_{U|V}(du)}{\int_\bU R_W(u,W)P_{U|V}(du)}
    + \indic_{\Omega\setminus\Omega^{\prime\prime}} P_{U|V}(A),  \label{eq:recbay}
\end{equation}
is a regular $(V,W)$-conditional distribution for $U$.  The mutual information
$I(U;(V,W))$ can be decomposed in the following way,
\begin{equation}
I(U;(V,W)) =  I(U;V) + \Eout I(U;W|V),  \label{eq:mutspl}
\end{equation}
where $I(U;W|V)$ is the {\em $V$-conditional mutual information} between $U$ and $W$,
\begin{equation}
I(U;W|V) := \cd(P_{UW|V}|P_{U|V}\otimes P_{W|V})
          = \Eout\left(\cd(P_{U|VW}\cond P_{U|V})\cond V\right). \label{eq:cmutinf}
\end{equation}
(The conditional mutual information is sometimes defined as the {\em average value}
of this quantity \cite{coth1}.)  Equation (\ref{eq:recbay}) can be used
{\em recursively} in estimation problems having sequences of conditionally independent
observations.  The information extracted from each observation in the sequence is then
associated with a ``local'' Bayesian problem, in which earlier observations enter only
through the ``local prior'' (the posterior derived from the earlier observations). The
decomposition (\ref{eq:mutspl}), (\ref{eq:cmutinf}) is valid whether or not $V$ and $W$
are $U$-conditionally independent.  However, in the absence of such conditional
independence it is not possible to interpret $\cd(P_{U|VW}\cond P_{U|V})$ as a
{\em local} information gain in this way.

Let $((X_t,Y_t),\,t\ge 0)$ be as described at the start of this section, and consider
the problem of estimating the {\em path} of $X$ from $Y$.  For any $0\le t\le s<\infty$,
let
\begin{equation}
Y_t^s := (Y_r-Y_t,\,r\in[t,s]).  \label{eq:yinc}
\end{equation}
Then $Y_0^t$ and $Y_t^s$ are $X$-conditionally independent, and we can use the above
methodology to identify the $Y_0^t$-conditional mutual information between $X$ and
$Y_t^s$, $I(X;Y_t^s|Y_0^t)$.  This is shown in section {\ref{se:qvar} to be related to
the quadratic variation of the posterior distribution in the Fisher metric.  The latter
is defined in terms of the mixed second derivative of the KL-divergence, and so
an appropriate state space for the nonlinear filter in this context is a subset of
$\cp(\cx)$ having a differentiable structure with respect to which the KL-divergence
admits such a derivative.  This is also desirable for the assessment of approximation
errors.  {\em Information Geometry} is the study of sets of probability measures having
such structures.

Information geometry is applied to nonlinear filtering in \cite{bhlg1} and the
references therein.  The posterior distributions for diffusion signal processes are
assumed, there, to have densities with respect to Lebesgue measure, whose square-roots
satisfy stochastic differential equations in $L^2(\R^m,\cb^m,{\rm Leb})$.  The induced
distance function between probability measures is the {\em Hellinger distance}.  The
coefficients of the filtering equation are projected in this sense onto the tangent
spaces of finite-dimensional exponential models, in order to obtain approximations to
filters. Information theoretic justification is given (when restricted to tangent
vectors corresponding to differentiable curves of square-root probability densities,
the $L^2$ norm corresponds to the Fisher metric), and comparisons are made with other
methods such as moment matching.  Although suitable for this purpose, the Hellinger
space cannot be used as an infinite-dimensional statistical manifold since the
KL-divergence is discontinuous at every point of it. (See the discussion at the end of
section 2 in \cite{newt4}.)  Furthermore, in common with the $L^2$ space of
densities, it has a boundary, which can create problems with numerical methods.

The local information gain of a nonlinear filter is connected with the notion of
{\em entropy production} in nonequilibrium statistical mechanics
\cite{mine2,newt2,newt3}.  The information geometric properties of nonlinear filters
are also, therefore, of interest in this context. 

The remainder of the paper is structured as follows.  Section \ref{se:infgeo} outlines
the main ingredients of information geometry and reviews the Hilbert manifold $M$,
which is used extensively in the sequel.  Section \ref{se:mvnlf} outlines a general
nonlinear filtering problem, and expresses the associated process of conditional
distributions as an It\^{o} process on $M$.  This allows the study of the quadratic
variation of the filter in the Fisher metric.  An $M$-valued evolution equation is
derived for a class of finite-dimensional exponential filters in section
\ref{se:finexp}.  The results of sections \ref{se:mvnlf} and \ref{se:finexp} are
formulated in terms of a set of hypotheses, some of which are not especially ripe.
Section \ref{se:examples} develops three examples in which they are satisfied.
Finally, section \ref{se:concl} makes some concluding remarks.

\section{Information Geometry} \label{se:infgeo}

We review the main ingredients of information geometry, by outlining the classical
finite-dimensional exponential model.  This is also used in section \ref{se:finexp}.
Let $(\bX,\cx,\mu)$ be a probability space on which are defined random variables
$(\xitil_i;\,i=1,\ldots,n)$ with the following properties: (i) the random variables
$(1,\xitil_1,\xitil_2,\ldots,\xitil_n)$ represent linearly independent elements of
$L^0(\mu)$, i.e.~$\mu(\alpha+\sum_i y^i\xitil_i=0)=1$ if and only if $\alpha=0$ and
$\R^n\ni y=0$;  (ii) $\Emu\exp(\sum_iy^i\xitil_i)<\infty$ for all $y$ in a
non-empty open subset $G\subseteq\R^n$.  For each $y\in G$, let $P_y$ be the
probability measure on $\cx$ with density
\[
\frac{dP_y}{d\mu} = \exp\left(\sum_iy^i\xitil_i-c(y)\right),
\]
where $c(y)=\log\Emu\exp(\sum_iy^i\xitil_i)$, and let $N:=\{P_y:\,y\in G\}$.  It
follows from (i) that the map $G\ni y\mapsto P_y\in N$ is a bijection.  Let
$\theta:N\rightarrow G$ be its inverse; then $(N,\theta)$ is an {\em exponential
statistical manifold} with an atlas comprising the single chart $\theta$.  We can think
of a {\em tangent vector} at $P\in N$ as being an equivalence class of differentiable
curves passing through $P$: two curves (expressed in coordinates),
$(\bfy(t)\in G:t\in(-\epsilon,\epsilon))$ and $(\bfz(t)\in G:t\in(-\epsilon,\epsilon))$
being equivalent at $P$ if $\bfy(0)=\bfz(0)=\theta(P)$ and $\dot{\bfy}(0)=\dot{\bfz}(0)$.
The {\em tangent space} at $P$, $T_PN$, is the linear space of all such tangent vectors,
and is spanned by the vectors $(\partial_i;\,i=1,\ldots,n)$, where $\partial_i$ is the
equivalence class containing the curve
$(\bfy_i(t):=\theta(P)+t\bfe_i,\,t\in(-\epsilon,\epsilon))$, and $\bfe_i^j$ is equal to
the Kr\"{o}necker delta.  The {\em tangent bundle} is the disjoint union
$TN := \cup_{P\in N}(P,T_PN)$, and admits the global chart
$\Theta:TN\rightarrow G\times\R^n$, where
$\Theta^{-1}(y,u):=(\theta^{-1}(y),u^i\partial_i)$.  If a function $f:N\rightarrow\R^k$
is differentiable, and $U\in T_PN$, then we write
\[
Uf = u^i\partial_i f
  := u^i\left.\frac{d}{dt}(f\circ\theta^{-1})(\bfy_i(t))\right|_{t=0}
   = u^i\frac{\partial (f\circ\theta^{-1})}{\partial y^i}(y),
\]
where $(y,u)=\Theta(P,U)=(\theta(P),U\theta)$, and we have used the Einstein summation
convention, that indices appearing once as a superscript and once as a subscript are
summed out.   

According to the Eguchi relations \cite{eguc1}, the mixed second derivative of the
KL-divergence defines the {\em Fisher metric} as a Riemannian metric on
$N$:  for any $P\in N$ and any $U,V\in T_PN$,
\[
\langle U, V \rangle_P := -UV\cd = g(P)_{i,j}u^iv^j,
\]
where $U$ and $V$ act on the first and second argument of $\cd$, respectively, and
\begin{equation}
g(P)_{i,j} := \langle \partial_i, \partial_j \rangle_P
  = \EP(\xitil_i-\EP\xitil_i)(\xitil_j-\EP\xitil_j),  \label{eq:fishmat}
\end{equation}
is the matrix form of the Fisher metric \cite{amar1}. The mixed third derivatives of
the KL-divergence define a pair of {\em covariant derivatives} on $N$ \cite{amar1}.
These give rise to notions of {\em curvature} of statistical manifolds, which are
important in the theory of asymptotic statistics \cite{amar1}.

The literature on information geometry is dominated by the study of finite-dimensional
manifolds of probability measures such as $(N,\theta)$.
The reader is referred to \cite{amar1,chen1} and the references
therein for further information.  In order to extend these ideas to infinite-dimensions
we need to choose a system of charts with respect to which the KL-divergence admits a
suitable number of derivatives.  It is clear from (\ref{eq:kldiv}) that the
smoothness properties of this divergence are closely connected with those of the
density $dQ/dP$ and its log (considered as elements of dual spaces of functions).  In
the series of papers \cite{cepi1,gipi1,piro1,pise1}, G.~Pistone and his
co-workers developed an infinite-dimensional exponential statistical manifold on an
abstract probability space $(\bX,\cx,\mu)$.  Probability measures in the manifold are
mutually absolutely continuous with respect to the reference measure $\mu$, and the
manifold is covered by the charts $s_P(Q)=\log dQ/dP-\EP\log dQ/dP$ for different
``patch-centric'' probability measures $P$.  These readily give $\log dQ/dP$ the
desired regularity, but require ranges that are subsets of exponential Orlicz
spaces in order to do the same for $dQ/dP$.  The exponential Orlicz manifold is a
natural extension of the finite-dimensional manifold $(N,\theta)$ described above;
it has a strong topology, under which the KL-divergence is of class $C^\infty$.

However, this approach is technically demanding and leads to manifolds that are larger
than needed in many applications.  Furthermore, the exponential Orlicz space is less
suited to the theory of stochastic differential equations than Hilbert space; the
latter is the natural setting for the $L^2$ theory of stochastic integration
\cite{daza1}.  An infinite-dimensional Hilbert manifold of ``finite-entropy''
probability measures, on which the KL-divergence is twice differentiable, is developed
in \cite{newt4}.  This uses a chart involving both the density, $dP/d\mu$, and
its log.  The finite entropy condition is natural in estimation problems where the
mutual information between the estimand and the observation is finite.  Banach
manifolds of finite-entropy measures, on which the KL-divergence admits higher
derivatives, are developed in \cite{newt5}.  That reference also develops
Hilbert and Banach manifolds of {\em finite measures} suitable for the ``un-normalised''
equations of nonlinear filtering.  We shall make extensive use of the Hilbert manifold
of \cite{newt4} in this paper; it is reviewed next.

\subsection{The Hilbert Manifold $M$} \label{se:manif}

For a probability space $(\bX, \cx, \mu)$, $M$ is the set of probability measures on
$\cx$ satisfying the following conditions:
\begin{enumerate}
\item[(M1)] $P$ is mutually absolutely continuous with respect to $\mu$;
\item[(M2)] $\Emu p^2 < \infty$;
\item[(M3)] $\Emu\log^2 p < \infty$.
\end{enumerate}
(We denote probability measures in $M$ by the upper-case letters $P$, $Q$, etc., and
their densities with respect to $\mu$ by the corresponding lower case letters, $p$,
$q$, etc.)  Let $\cl^0(\bX,\cx)$ be the set of real-valued random variables on $\bX$,
and $\cl^2(\bX,\cx,\mu)$ the subset of square-integrable random variables.  Let $H$ be
the Hilbert space of equivalence classes of centred elements of $\cl^2(\bX,\cx,\mu)$
(those having zero mean), and let $\Lambda:\cl^0(\bX,\cx)\rightarrow H$ be defined by
\begin{eqnarray}
\Lambda f & \ni &  f-\Emu f \quad{\rm if\ } f\in \cl^2(\bX,\cx,\mu), \nonumber \\
\Lambda f &  =  & 0 \quad {\rm otherwise}. \label{eq:lamdef}
\end{eqnarray}
Let $m,e:M\rightarrow H$ be defined as follows:
\begin{equation}
m(P)    = \Lambda p \quad {\rm and}\quad
e(P)    = \Lambda\log p.  \label{eq:medef}
\end{equation}
Variants of these are used in finite-dimensional information geometry as coordinate
maps for {\em mixture} and {\em exponential} models, respectively \cite{amar1}.
However, in the present context their images $m(M)$ and $e(M)$ are typically not open
subsets of $H$ \cite{newt4}; so, even though they are injective, $m$ and $e$ cannot be
used as charts for $M$.  On the other hand the map $\phi:M\rightarrow H$ defined by
their sum,
\begin{equation}
\phi(P) = m(P) + e(P) = \Lambda(p+\log p), \label{eq:phidef}
\end{equation}
is a bijection \cite{newt4}.  $(M,\phi)$ is a Hilbert manifold with an atlas comprising
a single chart.

Although not themselves charts, the maps $m$ and $e$ provide useful representations of
elements of $H$ since they are {\em bi-orthogonal}.  The simplest manifestation of
this property is the identity
\begin{equation}
\relpq + \relqp = \langle m(P)-m(Q), e(P)-e(Q)\rangle_H.
\end{equation}
It thus follows from the non-negativity of the KL-divergence that
\begin{equation}
\|m(P)-m(Q)\|_H^2 + \|e(P)-e(Q)\|_H^2 \le \|\phi(P)-\phi(Q)\|_H^2; \label{eq:phidif}
\end{equation}
in particular, $m\circ\phi^{-1}$ and $e\circ\phi^{-1}$ are Lipschitz continuous.
Furthermore
\begin{equation}
\relpq + \relqp \le \half \|\phi(P)-\phi(Q)\|_H^2.
\end{equation}
The inverse map $\phi^{-1}:H\rightarrow M$ is given by
\[
\frac{d\phi^{-1}(a)}{d\mu}(x) = \psi(\atil(x)+Z(a)),
\]
where $\psi:\R\rightarrow (0,\infty)$ is the inverse of the function
$(0,\infty)\ni z\mapsto z+\log z\in\R$, $\atil$ is any function in the equivalence class
$a$, and $Z:H\rightarrow\R$ is the unique function for which $\Emu\psi(\atil+Z(a))=1$
for all $a\in H$.  $Z$ is (Fr\'{e}chet) differentiable with derivative
\begin{equation}
DZ_au = -\frac{\Emu\psi^\prime(\atil+Z(a))\util}{\Emu\psi^\prime(\atil+Z(a))},
        \label{eq:Zder}
\end{equation}
where $\util$ is any function in the equivalence class $u$ \cite{newt4}.

A tangent vector $U$ at $P\in M$ is an equivalence class of differentiable curves at
$P$.  We denote the tangent space at $P$ by $T_PM$, and the tangent bundle by $TM$.
The latter admits the global chart $\Phi:TM\rightarrow H\times H$, where
$\Phi(P,U) = (\bfa(0),\dot{\bfa}(0))$, and $(\bfa(t),\,t\in(-\epsilon,\epsilon))$ is
any differentiable curve in the equivalence class $U$ (expressed in terms of the chart
$\phi$).  If $f:M\rightarrow Y$ is a map with range $Y$ (a Banach space) and the map
$f\circ\phi^{-1}:H\rightarrow Y$ is (Fr\'{e}chet) differentiable, then we write
\[
Uf := \left.\frac{d}{dt}(f\circ\phi^{-1})(\bfa(t))\right|_{t=0}
    = D(f\circ\phi^{-1})_au,
\]
where $(a,u)=\Phi(P,U)=(\phi(P),U\phi)$.  A weaker notion of {\em $d$-differentiability}
is defined in \cite{newt4}.  The map $f:M\rightarrow Y$ is $d$-differentiable
if, for any $P\in M$, there exists a continuous linear map
$d(f\circ\phi^{-1})_a:H\rightarrow Y$ such that
\[
\left.\frac{d}{dt}(f\circ\phi^{-1})(\bfa(t))\right|_{t=0} = d(f\circ\phi^{-1})_au,
\]
for all differentiable curves $\bfa$ in the equivalence class $U$.  We then write
$Uf=d(f\circ\phi^{-1})_au$.

The KL-divergence, $\cd:M\times M\rightarrow[0,\infty)$, is Fr\'{e}chet
differentiable in each argument, and both derivatives are $d$-differentiable in the
remaining argument \cite{newt4}.  We can use this fact, together with the Eguchi
relations \cite{eguc1}, to define the Fisher metric on $T_PM$: for any $P\in M$ and
$U,V\in T_PM$,
\begin{equation}
\langle U,V\rangle_P := -UV\cd
  = \Emu\frac{p}{(1+p)^2}(\util+DZ_au)(\vtil+DZ_av),  \label{eq:fisherm}
\end{equation}
where $U$ and $V$ act on the first and second arguments of $\cd$, respectively,
$a=\phi(P)$, $u=U\phi$, $v=V\phi$, $\util$ is any function in the equivalence
class $u$, and $\vtil$ is any function in the equivalence class $v$.
$(T_PM,\langle\fndot,\fndot\rangle_P)$ is an {\em inner product space}, whose norm
admits the bound: $\|U\|_P \le \|u\|_H$.  However, since the Fisher norm is not (in
general) equivalent to the model space norm, $(T_PM,\langle\fndot,\fndot\rangle_P)$ is
not a Hilbert space. $(M,\langle\fndot,\fndot\rangle_P)$ is a {\em pseudo}-Riemannian
manifold rather than a Riemannian manifold.

$H$-valued stochastic processes play a major role in what follows.  In order to
ensure that they have suitable measurability properties, we introduce the following
additional hypothesis.  This is satisfied, for example, if $\bX$ is a complete
separable metric (Polish) space, and $\cx$ is its Borel $\sigma$-algebra.
\begin{enumerate}
\item[(M4)] $H$ is separable.
\end{enumerate}

\begin{lemma}  \label{le:Pettis}
If $(\bZ,\cz)$ is a measurable space, and $f:\bZ\times\bX\rightarrow \R$ is
jointly measurable, then the map $\bZ\ni z\mapsto\Lambda f(z,\fndot)\in H$
is $\cz$-measurable.
\end{lemma}

\begin{proof}
Let $B:=\{z\in\bZ:f(z,\fndot)\in\cl^2(\bX,\cx,\mu)\}$ then, according to Tonelli's
theorem, $B\in\cz$.  Fubini's theorem shows that, for any $g\in\cl^2(\bX,\cx,\mu)$,
the function $\bZ\ni z\mapsto \indic_B(z)\Emu f(z,\fndot)g \in \R$ is $\cz$-measurable,
i.e.~the map $\bZ\ni z\mapsto\Lambda f(z,\fndot)\in H$ is {\em weakly} $\cz$-measurable.
The statement of the lemma follows from (M4) and Pettis's theorem.
\end{proof}
 
Under (M4), $H$ is of countable dimension, and so admits a complete orthonormal basis
$(\eta_i\in H,\,i\in\N)$.  An element $(P,U)\in TM$ thus has the coordinate
representation $((a^i,u^j),\,i,j\in\N)$ where $a^i=\langle \phi(P),\eta_i\rangle_H$ and
$u^j=\langle U\phi,\eta_j\rangle_H$.  In particular, any $U\in T_PM$ admits the
representation $U = u^jD_j$, where $(P,D_j)=\Phi^{-1}(\phi(P),\eta_j)$.
For any $P\in M$ and $i,j\in\N$, let
\begin{equation}
G(P)_{i,j} := \langle D_i, D_j\rangle_P.  \label{eq:cmetric}
\end{equation}
Then it follows from the domination of $\|U\|_P$ by $\|u\|_H$ that, for any
$U,V\in T_PM$,
\begin{equation}
\langle U,V\rangle_P = G(P)_{i,j}u^iv^j, \label{eq:metrep}
\end{equation}
in the sense that both series are absolutely convergent, and the result does not depend
on the order in which the limits are taken.

\section{The $M$-Valued Nonlinear Filter} \label{se:mvnlf}

We consider a general nonlinear filtering problem as outlined in section \ref{se:intro},
in which all random quantities are defined on a complete probability space
$(\Omega,\cf,\PR)$.  The signal space $\bX$ is a complete separable metric space,
$\cx$ is its Borel $\sigma$-algebra, and $\mu$ is a reference probability measure on
$\cx$.  $(M,\phi)$ is the associated Hilbert manifold, as described in section
\ref{se:manif}.  We shall assume that $X$ has right-continuous sample paths with left
limits at all $t\in(0,\infty)$, and that the distribution of $X_t$, $P_t$, has a
density with respect to $\mu$ satisfying the Kolmogorov forward equation
\begin{equation}
\frac{\partial p_t}{\partial t}
  = \ca_t p_t \quad{\rm for\ }t\in[0,\infty), \label{eq:kolfor}
\end{equation}
where $(\ca_t,\,t\ge 0)$ is a family of linear operators on an appropriate class
of functions $f:\bX\rightarrow\R$.

\begin{example} \label{ex:discsig}
$\bX=\{1,2,\ldots, m\}$, $X$ is a time-homogeneous Markov jump process with rate
matrix $A$, $\mu$ is mutually absolutely continuous with respect to the counting
measure (with Radon-Nikodym derivative $r$), and $h_t(=h)$ does not depend on $t$.
In this case
\[
(\ca_tp)(x)
  = (\ca p)(x)
  = \frac{1}{r(x)}\sum_{\xtil\in\bX}A_{\xtil x}rp(\xtil) \quad{\rm for\ all\ }t.
\]
\end{example}

\begin{example} \label{ex:diffsig}
$\bX=\R^m$, $X$ is a time-homogeneous multidimensional diffusion process with
suitably regular drift vector $b$ and diffusion matrix $a$, $\mu$ is mutually
absolutely continuous with respect to Lebesgue measure (with Radon-Nikodym
derivative $r$), and $h_t(=h)$ does not depend on $t$.  In this case
\[
\ca_tp
  =  \ca p
  =  \frac{1}{2r}\frac{\partial^2}{\partial x^i\partial x^j}(a^{ij}rp)
     - \frac{1}{r}\frac{\partial}{\partial x^i}(b^irp) \quad{\rm for\ all\ }t.
\]
\end{example}

The set-up is sufficiently general to include {\em path estimators}.

\begin{example} \label{ex:pathest}
$\bX=C([0,\infty);\R^m)$, $X_t=(\Xtil_s,\,s\ge 0)$ for all $t$, where $\Xtil$ is the
diffusion process of Example \ref{ex:diffsig}, and $h_t(x)=\htil(x_t)$ for some
$\htil:\R^m\rightarrow\R^d$.  In this case $\ca_t=0$ for all $t$.
\end{example}

Let $(\cy_t\subset\cf,\,t\ge 0)$ be the filtration generated by the observation
process $Y$, augmented by the $\PR$-null sets of $\cf$, and let $\cp_Y$ be the
$\sigma$-algebra of $(\cy_t)$-predictable subsets of $[0,\infty)\times\Omega$.
We assume that:
\begin{enumerate}
\item[(F1)] $P_0\in M$.
\item[(F2)] For any $T<\infty$, $\int_0^T\Eout|h_t(X_t)|^2dt< \infty$.
\item[(F3)] There exists, on the product space $\Omega\times\bX$, a
  $\cp_Y\times\cx$-measurable, $(0,\infty)$-valued process $(\pi_t,\, t\ge 0)$, for
  which
  \[
  \PR\left(\int_\bX\pi_t(\fndot,x)\mu(dx)=1 {\rm\ for\ all\ }t\ge 0\right) = 1.
  \]
  For any $t$ and any $A\in\cx$, $\PR(X_t\in A\cond \cy_t)=\Pi_t(A)$, where
  \begin{equation}
  \Pi_t(A) := \int_A\pi_t(\fndot,x)\mu(dx).  \label{eq:f22}
  \end{equation}
\item[(F4)] $\PR(\pi_t\in {\rm Dom}\ca_t {\rm\ for\ all\ }t\ge 0)=1$, the process
  $(\ca_t\pi_t,\,t\ge 0)$ is $\cp_Y\times\cx$-measurable and, for any $T<\infty$,
  \begin{equation}
  \PR\left(\int_0^T\sqrt{\Emu(1+\pi_t^{-1})^2(\ca_t\pi_t)^2}\,dt < \infty\right)
    = 1. \label{eq:f41}
  \end{equation}
\item[(F5)] For any $T<\infty$,
  \begin{eqnarray}
  \PR\left(\int_0^T\sqrt{\Emu|h_t-\hhbar_t|^4}\,dt < \infty\right)
    & = & 1, \label{eq:f43} \\
  \PR\left(\int_0^T\Emu(\pi_t+1)^2|h_t-\hhbar_t|^2dt < \infty\right)
    & = & 1, \label{eq:f42}
  \end{eqnarray}
  where
  \begin{equation}
  \hhbar_t := \left\{\begin{array}{l}
                \Emu\pi_t h_t \quad {\rm if\ }\Emu\pi_t|h_t| < \infty \\
                0 \quad {\rm otherwise}.
              \end{array} \right. \label{eq:hbardef}
  \end{equation}
\item[(F6)] For almost all $x$, $(\pi_t,\,t\ge 0)$ satisfies the following It\^{o}
  equation on $\Omega$:
  \begin{equation}
  \pi_t = p_0 + \int_0^t \ca_s\pi_s \,ds
           + \int_0^t \pi_s(h_s-\hhbar_s)^*d\nu_s, \label{eq:pifilt}
  \end{equation}
  where $(\nu_t,\,t\ge 0)$ is the {\em innovations process},
  \begin{equation}
  \nu_t := Y_t - \int_0^t \hhbar_s\,ds.  \label{eq:innov}
  \end{equation}
\end{enumerate}

\begin{remark} \label{re:Frem}
\begin{enumerate}
\item[(i)] Because of (F2), $(\nu_t,\,t\ge 0)$ is a $d$-dimensional $(\cy_t)$-Brownian
  motion \cite{lish1}.
\item[(ii)] In the context of Example \ref{ex:discsig}, (\ref{eq:pifilt}) becomes a
  system of stochastic ordinary differential equations derived independently by Wonham
  \cite{wonh1} and Shiryayev \cite{shir1}.  In the context of Example \ref{ex:diffsig},
  it becomes a stochastic partial differential equation known as the Kushner-Stratonovich
  equation \cite{kush1}. See \cite{crro1} for a variety of conditions under
  which nonlinear filters admit the representation (\ref{eq:pifilt}).
\end{enumerate}
\end{remark}

The intention here is to develop an $M$-valued representation for the process $\Pi$
of (\ref{eq:f22}). With this in mind, we introduce the following $H$-valued processes
\begin{equation}
u_t     := \Lambda(1+\pi_t^{-1})\ca_t\pi_t \quad{\rm and}\quad
\zeta_t := \half\Lambda|h_t-\hhbar_t|^2, \label{eq:uzetdef}
\end{equation}
where $\Lambda$ is as defined in (\ref{eq:lamdef}), and the following
$L(\R^d,H)$-valued process
\begin{equation}
v_t := \Lambda(\pi_t+1)(h_t-\hhbar_t)^*.  \label{eq:vdef}
\end{equation}

\begin{proposition} \label{pr:mvalfil}
Suppose that $(X,Y)$ satisfies (F1--F6), and $\Pi$, $u$, $\zeta$ and $v$ are as defined
in (\ref{eq:f22}), (\ref{eq:uzetdef}) and (\ref{eq:vdef}).  Then
\begin{enumerate}
\item[(i)] $\PR(\Pi_t\in M {\rm\ for\ all\ }t\ge 0)=1$;
\item[(ii)] $(\phi(\Pi_t),\,t\ge 0)$ satisfies the following (infinite-dimensional)
  It\^{o} equation
  \begin{equation}
  \phi(\Pi_t) = \phi(P_0) + \int_0^t(u_s-\zeta_s)\,ds
                + \int_0^t v_s\,d\nu_s. \label{eq:hfilt}
  \end{equation}
\end{enumerate}
\end{proposition}

\begin{remark} \label{re:idint}
\begin{enumerate}
\item[(i)] The first integral in (\ref{eq:hfilt}) is a Bochner integral, and the second
  is an It\^{o} integral.  The stochastic calculus of Hilbert-space-valued
  semimartingales is developed pedagogically in \cite{daza1}.  In the general case,
  the stochastic integral is defined for a Hilbert-space-valued Wiener process and
  the stochastic integrand is a Hilbert-Schmidt-operator-valued process. In the present
  context $\nu$ is of finite dimension, and so, if $\R^d$ is equipped with the
  Euclidean inner product, {\em any} element of $L(\R^d,H)$ is a Hilbert-Schmidt
  operator.  For any $\rho,\sigma\in L(\R^d,H)$, the associated inner product is
  \begin{equation}
  \langle\rho,\sigma\rangle_{HS}
    := \sum_{k=1}^d\langle\rho \bfe_k,\sigma \bfe_k\rangle_H, \label{eq:hsinprod}
  \end{equation}
  where $(\bfe_k,\,1\le k\le d)$ is any orthonormal basis in $\R^d$.
\item[(ii)] Although natural and inclusive, (F3--F6) are not particularly ripe.  We
  develop some examples in which they are satisfied in section \ref{se:examples}.
\item[(iii)] The case $h=0$ provides an $M$-valued representation for the marginal
  distribution $(P_t,\,t\ge 0)$.
\end{enumerate}
\end{remark}

\begin{proof}
According to (F6) there exists an $F\in\cx$ with $\mu(F)=1$ such that $\pi(\fndot,x)$
satisfies (\ref{eq:pifilt}) on $\Omega$ for all $x\in F$.  It\^{o}'s rule shows that,
for any such $x$, $(\pi_t+\log\pi_t,\,t\ge 0)$ satisfies the following It\^{o} equation
on $\Omega$:
\begin{equation}
\pi_t+\log\pi_t = p_0+\log p_0 + I_t + J_t,  \label{eq:pplp}
\end{equation}
where
\begin{eqnarray}
I_t & := & \int_0^t(\util_s-\zetatil_s)\,ds,
           \quad \util_t:=\indic_F(1+\pi_t^{-1})\ca_t\pi_t,
           \quad \zetatil_t := \half\indic_F|h_t-\hhbar_t|^2, \nonumber \\
J_t & := & \int_0^t \vtil_sd\nu_s,
           \quad\qquad\ \vtil_t:=\indic_F(\pi_t+1)(h_t-\hhbar_t)^*.
           \label{eq:IJdef}
\end{eqnarray}
The Fubini theorem and Cauchy-Schwartz inequality show that, for any
$T<\infty$,
\begin{eqnarray*}
\Emu\left(\int_0^T|\util_t-\zetatil_t|\,dt\right)^2
  &   =   & \int_0^T\int_0^T\Emu|\util_t-\zetatil_t||\util_s-\zetatil_s| \,ds\,dt \\
  &  \le  & \left(\int_0^T\sqrt{\Emu(\util_t-\zetatil_t)^2}\,dt\right)^2,
\end{eqnarray*}
and so it follows from (\ref{eq:f41}) and (\ref{eq:f43}) that
\begin{equation}
\PR\big(I_t\in\cl^2(\bX,\cx,\mu){\rm \ for\ all\ }t\ge 0\big) = 1.  \label{eq:Iinl2}
\end{equation}
For any $m\in\N$ and any $T<\infty$, let
\begin{equation}
\tau_m:=\inf\{t>0:\,K_t\ge m\}\wedge T
  \quad{\rm where}\quad K_t := \int_0^t\Emu|\vtil_s|^2\,ds. \label{eq:taukdef}
\end{equation}
According to (\ref{eq:f42}), $\PR(K_T<\infty)=1$, $\tau_m$ is a $(\cy_t)$-stopping time,
and $(J_{t\wedge\tau_m}, t\ge 0)$ is a continuous martingale on $\Omega$, for almost all
$x$. According to the stochastic Fubini theorem (Theorem 4.18 in \cite{daza1}),
$J_{t\wedge\tau_m}$ is $\PR\otimes\mu$-measurable for each $t$, and so
$\sup_{[0,T]}J_{t\wedge\tau_m}^2$ is also $\PR\otimes\mu$-measurable. 
Applying Doob's $L^2$ inequality and then integrating with respect to $\mu$, we obtain
\[
\Emu\Eout\sup_{t\in[0,T]} J_{t\wedge\tau_m}^2 \le 4\Emu\Eout J_{\tau_m}^2
  = 4\Eout K_{\tau_m} \le 4m.
\]
Now $\tau_m=T$ for all $m\ge K_T$, and so
\begin{equation}
\PR\left(\Emu\sup_{t\in[0,T]} J_t^2 < \infty\right) = 1; \label{eq:supJbnd}
\end{equation}
in particular,
\begin{equation}
\PR\big(J_t\in\cl^2(\bX,\cx,\mu){\rm \ for\ all\ }t\ge 0\big)=1.  \label{eq:Jinl2}
\end{equation}
Now $\inf_{y\in(0,\infty)}y\log y=-1/e$, and so
$\pi_t^2 + \log^2\pi_t \le (\pi_t+\log\pi_t)^2 + 2/e$.  Part (i) thus follows from
(F1), (\ref{eq:pplp}), (\ref{eq:Iinl2}) and (\ref{eq:Jinl2}), as does the fact that
\[
\PR\big(\phi(\Pi_t)=\phi(P_0)+\Lambda I_t+\Lambda J_t{\rm\ for\ all\ }t\big) = 1.
\]

According to (F3), (F4) and Lemma \ref{le:Pettis}, $\phi(\Pi)$, $\hhbar$, $u$, $\zeta$
and $v$ are all $\cp_Y$-measurable. Furthermore, it follows from
(\ref{eq:f41}--\ref{eq:f42}) that
\[
\PR\left(\int_0^t\|u_s-\zeta_s\|_H\,ds+\int_0^t\|v_s\|_{HS}^2\,ds < \infty
  {\rm \ for\ all\ }t\ge 0\right) = 1,
\]
and so the integrals on the right-hand side of (\ref{eq:hfilt}) are well defined.
(See section 4.2 in \cite{daza1}.)  It remains to show that the operator
$\Lambda$ commutes with the operations of ordinary and stochastic integration in
(\ref{eq:IJdef}).

Let $(a_t,\,t\ge 0)$ be the (continuous) $H$-valued process on the right-hand
side of (\ref{eq:hfilt}), let $(\eta_i,\,i\in\N)$ be a complete orthonormal basis for
$H$ and, for each $i$, let $\etatil_i\in\cl^2(\bX,\cx,\mu)$ be a function in the
equivalence class $\eta_i$.  It follows from the infinite-dimensional It\^{o} rule
(Theorem 4.17 in \cite{daza1}) that
\begin{eqnarray}
\langle\eta_i,a_t\rangle_H
  & = & \langle\eta_i,\phi(P_0)\rangle_H
        + \int_0^t\langle\eta_i,u_s-\zeta_s\rangle_H\,ds
        + \int_0^t\langle\eta_i,v_sd\nu_s\rangle_H \nonumber \\
  & = & \Emu\etatil_i(p_0+\log p_0)
        + \int_0^t\Emu\etatil_i(\util_s-\zetatil_s)\,ds
        + \int_0^t\Emu\etatil_i\vtil_s\,d\nu_s,\qquad \label{eq:compexp}
\end{eqnarray}
where, in the second step, we have used the fact (derived from
(\ref{eq:f41}--\ref{eq:f42})) that
$\util_t, \zetatil_t, \vtil_ty \in\cl^2(\bX,\cx,\mu)$ for all $y\in\R^d$, for almost all
$(t,\omega)$.  The Cauchy-Schwartz inequality shows that, for any $m\in\N$ and
any $T<\infty$,
\[
\Emu\left(\Eout\int_0^{\tau_m}|\etatil_i\vtil_s|^2\,ds\right)^{1/2}
   =   \Emu|\etatil_i|\left(\Eout\int_0^{\tau_m}|\vtil_s|^2\,ds\right)^{1/2}
  \le  \sqrt{\Eout K_{\tau_m}} \le \sqrt{m},
\]
where $\tau_m$ and $K_t$ are as defined in (\ref{eq:taukdef}), and so, according to
the stochastic Fubini theorem (Theorem 4.18 in \cite{daza1})  
\[
\int_0^{t\wedge\tau_m}\Emu\etatil_i\vtil_s \,d\nu_s
  = \Emu\etatil_i J_{t\wedge\tau_m}.
\]
Since $m$ and $T$ are arbitrary, this is also true if $t\wedge\tau_m$ is replaced by
any $t\ge 0$.  According to (\ref{eq:supJbnd}) and the dominated convergence theorem
$\Lambda J$ is continuous; so
\[
\PR\left(\int_0^t\Emu\etatil_i\vtil_s \,d\nu_s = \Emu\etatil_i J_t
  {\rm\ for\ all\ }t\ge 0\right) = 1.
\]
Applying this and the Fubini theorem to (\ref{eq:compexp}), we obtain
\begin{eqnarray*}
\langle\eta_i,a_t\rangle_H
  & = & \Emu\etatil_i(p_0+\log p_0) + \Emu\etatil_i I_t + \Emu\etatil_i J_t \\
  & = & \Emu\etatil_i(\pi_t+\log \pi_t) \\
  & = & \langle\eta_i,\phi(\Pi_t)\rangle_H,
\end{eqnarray*}
where we have used (\ref{eq:pplp}) in the second step, and part (i) in the third step.
This completes the proof of part (ii).
\end{proof}

\subsection{Quadratic Variation} \label{se:qvar}

Suppose that (F1--F6) hold, and let $(\phi(\Pi_t)^i,\,i\in\N)$ be the coordinate
representation for $\Pi_t$ in terms of a complete orthonormal basis,
$(\eta_i,\,i\in\N)$, for $H$.  Then $\phi(\Pi_t)^i$ is equal to the right-hand
side of (\ref{eq:compexp}).  The components $\{(\phi(\Pi_t)^i,\cy_t),\,i=1,2,\ldots\}$
form a system of real-valued semimartingales with quadratic co-variations
\begin{equation}
[\phi(\Pi)^i,\phi(\Pi)^j]_t
  = \int_0^t \sum_k\langle \eta_i,v_se_k\rangle_H\langle \eta_j,v_se_k\rangle_H\,ds.
    \label{eq:covar}
\end{equation}
We define the {\em $M$-intrinsic quadratic variation} of the process $\Pi$ (its
quadratic variation in the Fisher metric) as follows:
\begin{eqnarray}
[\Pi]_t
  & := & \int_0^t G(\Pi_s)_{i,j} \,d[\phi(\Pi)^i,\phi(\Pi)^j]_s \nonumber \\
  &  = & \int_0^t \sum_k G(\Pi_s)_{i,j}\langle \eta_i,v_se_k \rangle_H
         \langle \eta_j,v_se_k \rangle_H\,ds \nonumber \\
  &  = & \int_0^t \sum_k \Emu\frac{\pi_s}{(1+\pi_s)^2}
         (\vtil_s e_k + DZ_{\phi(\Pi_s)}v_se_k)^2  \,ds \nonumber \\
  &  = & \int_0^t \Emu\pi_s|h_s-\hhbar_s|^2\,ds, \label{eq:rqvfm}
\end{eqnarray}
where $G$ is as defined in (\ref{eq:cmetric}), and we have used (\ref{eq:metrep}) and
(\ref{eq:fisherm}) in the third step, and (\ref{eq:Zder}) in the final step.  The
final integrand here is the $\cy_s$-conditional mean-square error for the filter's
estimate of $h_s(X_s)$.  The average value of the final integral is known to be related
to the mutual information $I(X;Y_0^t)$ \cite{dunc1}.  We refine this result here,
characterising the $\cy_t$-conditional variant of this mutual information, as defined
in (\ref{eq:cmutinf}).

\begin{proposition} \label{pr:qvar}
For any $0\le t\le s<\infty$,
\begin{equation}
I(X;Y_t^s|Y_0^t) = \half\Eout([\Pi]_s-[\Pi]_t\cond\cy_t), \label{eq:qvsup}
\end{equation}
where $Y_t^s$ is as defined in (\ref{eq:yinc}).
\end{proposition}

\begin{proof}
Let $D([0,\infty);\bX)$ be the Skorohod space of right-continuous, left-limit maps
$\theta:[0,\infty)\rightarrow\bX$, and let
$H:[0,s]\times D([0,\infty);\bX)\rightarrow\R^d$ be defined by
\[
H_t(\theta) = \left\{\begin{array}{l}
   h_t(\theta_t) \quad{\rm if\ } \int_0^s|h_t(\theta_t)|^2dt < \infty \\
   0 \quad {\rm otherwise}. \end{array}\right.
\]
Fubini's theorem and (F2) show that $\PR(H_r(X)=h_r(X_r){\rm\ for\ all\ }r\in[0,s])=1$. 
Let $\rho:[0,s]\times\Omega\times D([0,\infty);\bX)\rightarrow[0,\infty)$ be defined by
\[
\rho_t(\fndot,\theta)
  = \exp\left(\int_0^t(H_r(\theta)-\hhbar_r)^*dB_r
      + \half\int_0^t|H_r(\theta)-\hhbar_r|^2dr\right),
\]
let $\cs$ be the Borel $\sigma$-algebra on $D([0,\infty);\bX)$, and let $P_X\in\cp(\cs)$
be the distribution of $X$.  Theorem 7.23 of \cite{lish1} shows that there exists
a regular $Y_0^t$-conditional distribution $P_{X|Y_0^t}:\Omega\rightarrow\cp(\cs)$, whose
density with respect to $P_X$ is $\rho_t$.  (The filtration $(\cf_t)$ of Theorem 7.23 is
here $\cf_t=X^{-1}(\cs)\vee\cy_t$.) So, from (\ref{eq:cmutinf}),
\begin{eqnarray*}
I(X;Y_t^s\cond \cy_t)
  & = & \Eout\left(\cd(P_{X|Y_0^s}\cond P_{X|Y_0^t})\cond\cy_t\right) \\
  & = & \Eout\left(\log\frac{\rho_s(\fndot,X)}{\rho_t(\fndot,X)}\,\big|\,\cy_t\right) \\
  & = & \half\Eout\left(\int_t^s|H_r(X)-\hhbar_r|^2 dr\,\big|\,\cy_t\right) \\
  & = & \half\Eout\left(\int_t^s\int_{D([0,\infty);\bX)}|h_r(\theta_r)-\hhbar_t|^2
        P_{X|Y_0^r}(d\theta)\,dr\,\big|\,\cy_t\right),
\end{eqnarray*}
which completes the proof.
\end{proof}

The solutions of stochastic differential equations driven by continuous semimartingales
are typically non-differentiable.  This is an important property of nonlinear filters
in continuous time. Over a {\em short} time interval $[t,s]$ the observation
process $Y$ introduces the small quantity of new information (\ref{eq:qvsup}) to the
filter. If the filter process $\Pi$ were differentiable, then the filter would ``know''
$\Pi_s$ to first-order accuracy at time $t$, thereby contradicting the novelty
of the information.  Proposition \ref{pr:qvar} takes this intuition further by
connecting the infinitesimal information gain with the quadratic variation of the
filter.

\section{Finite-Dimensional Exponential Filters} \label{se:finexp}

Let $(N,\theta)$ be the finite-dimensional exponential manifold outlined in section
\ref{se:infgeo}, with the stronger conditions that $\Emu\xitil_i^2<\infty$ for all $i$,
and $\Emu\exp(2\sum_iy^i\xitil_i)<\infty$ for all $y\in G$; let $\xi_i:=\Lambda\xitil_i$
for $i=1,\ldots,n$.  Theorem 5.1 in \cite{newt4} shows that $N$ is a
$C^\infty$-embedded submanifold of $M$.  We consider a nonlinear filtering problem, as
developed in section \ref{se:mvnlf}, fulfilling (F1--F6) and the following additional
hypotheses.
\begin{enumerate}
\item[(F7)] $\PR(\Pi_t\in N, {\rm\ for\ all\ }t\ge 0)=1$.
\item[(F8)] For any $t$ and any $P\in N$, there exists at least one measurable
function $p:\bX\rightarrow(0,\infty)$ lying in the domain of $\ca_t$ for which
$dP/d\mu=p$.  For any $p$ with this property, $p^{-1}\ca_tp\in\cl^2(\bX,\cx,\mu)$ and
$\Lambda p^{-1}\ca_tp\in\spanof\{\xi_i\}$.  For any $p,\ptil$ with this property
$\mu(\ca_tp=\ca_t\ptil)=1$.
\item[(F9)] For any $t$ and any $k$, $\Lambda h_t^k\in\spanof\{\xi_i\}$ and
$\Lambda |h_t|^2\in\spanof\{\xi_i\}$.
\end{enumerate}
Let $\bfU,\bfV_k:[0,\infty)\times N\rightarrow TN$ be vector fields on $N$ with
$e$-representations
\begin{eqnarray}
\bfu_{e,t}(P)
  & := & \Lambda p^{-1}\ca_tp = \bfu_{\theta,t}^i(P)\xi_i \label{eq:vecfiel1} \\
\bfv_{e,k,t}(P)
  & := & \bfv_{e,k,t} = \Lambda h_t^k = \bfv_{\theta,k,t}^i\xi_i, \label{eq:vecfiel2}
\end{eqnarray}
where $p$ is as in (F8), and $\bfu_{\theta,t}(P)$ and $\bfv_{\theta,k,t}$ are the
$\theta$-representations.  Since $\bfv_{e,k,t}$ does not depend on $P$,
$\bfV_{k,t}\in C^\infty(N;TN)$ for all $t$.  We assume, further, that
\begin{enumerate}
\item[(F10)] $\bfU_t\in C^1(N;TN)$ for all $t$.
\end{enumerate}
Consider the following intrinsic Stratonovich equation on $N$:
\begin{equation}
\circ dQ_t
  = \left(\bfU_t(Q_t)
    - \half\sum_{k=1}^d\nabla_{\bfV_{k,t}}^{(-1)}\bfV_{k,t}(Q_t)\right)\,dt
    + \bfV_{k,t}(Q_t)\circ d\Btil_t^k, \label{eq:fdfil}
\end{equation}
where $\nabla^{(-1)}$ is Amari's $-1$-covariant derivative \cite{amar1}, and
$(\Btil_t,t\ge 0)$ is a $d$-vector Brownian motion.

\begin{proposition} \label{pr:expnlf}
Suppose that $(X,Y)$ satisfies (F1--F6), and (F7--F10) with respect to $N$, and that
$\bfU$ and $\bfV_k$ are as defined in (\ref{eq:vecfiel1},\ref{eq:vecfiel2}).  Then
(\ref{eq:fdfil}) has a strong solution
$\Psi:N\times C([0,\infty);\R^d)\rightarrow C([0,\infty);N)$, and $\Pi=\Psi(P_0,\nu)$.
\end{proposition}

\begin{proof}
As in the proof of Proposition \ref{pr:mvalfil}, we can apply It\^{o}'s rule to
(\ref{eq:pifilt}) to show that, for all $x\in F$, $(\log\pi_t,\,t\ge 0)$
satisfies the following It\^{o} equation on $\Omega$:
\[
\log\pi_t = \log p_0 + \int_0^t\indic_F(\pi_s^{-1}\ca_s\pi_s-\zetatil_s)\,ds,
            + \int_0^t \indic_F(h_s-\hhbar_s)^*d\nu_s,
\]
where $\zetatil$ is as defined in (\ref{eq:IJdef}).  This can be ``lifted'' to an
$H$-valued equation in the same way that (\ref{eq:pplp}) was lifted to
(\ref{eq:hfilt}) in the proof of Proposition \ref{pr:mvalfil}.  The resulting equation
is:
\begin{eqnarray}
e(\Pi_t)
  & = & e(P_0)
        + \int_0^t(\Lambda\pi_s^{-1}\ca\pi_s-\zeta_s)\,ds
        + \int_0^t \Lambda h_s^*\,d\nu_s \nonumber \\
  & = & e(P_0) + \int_0^t (\bfu_{e,s}(\Pi_s)-\bfw_{e,s}(\Pi_s))\,ds
        + \int_0^t \bfv_{e,k,s}\,d\nu_s^k, \label{eq:efil}
\end{eqnarray}
where $\bfw_{e,t}(P)\,(:=\Lambda|h_t-\EP h_t|^2/2)$ is the $e$-representation of a
time-dependent vector field $\bfW:[0,\infty)\times N\rightarrow TN$.  For any $t$ such
that $|h_t|\in\cl^2(\bX,\cx,\mu)$,
$\EP h_t^k=\langle m(P),\Lambda h_t^k\rangle_H+\Emu h_t^k$, and so, according to
Theorem 5.1 in \cite{newt4}, $W_t\in C^\infty(N;TN)$ for all $t$.

The Christoffel symbols for $\nabla^{(-1)}$ can be found from the Eguchi relations
\cite{amar1,eguc1}:
\begin{eqnarray*}
\Gamma_{ij}^{(-1),l}(P)
  & = & -g^{lm}(P)\frac{\partial^3}{\partial y^i\partial y^j\partial \ytil^m}
        \cd(\theta^{-1}(y)\,|\,\theta^{-1}(\ytil))\big|_{\ytil=y=\theta(P)} \\
  & = & g^{lm}(P)\EP(\xitil_i-\EP\xitil_i)(\xitil_j-\EP\xitil_j)(\xitil_m-\EP\xitil_m),
\end{eqnarray*}
where $g^{lm}(P)$ is the $(l,m)$ element in the inverse of the Fisher matrix in
$\theta$-coordinates, $g$, as defined in (\ref{eq:fishmat}).  So
\begin{eqnarray*}
\nabla_{\bfV_{k,t}}^{(-1)}\bfV_{k,t}(P)
  & = & \Gamma_{ij}^{(-1),l}(P)\bfv_{\theta,k,t}^i\bfv_{\theta,k,t}^j \partial_l \\
  & = & g^{lm}(P)\EP\bfv_{\theta,k,t}^i(\xitil_i-\EP\xitil_i)\bfv_{\theta,k,t}^j
        (\xitil_j-\EP\xitil_j)(\xitil_m-\EP\xitil_m) \partial_l \\
  & = & g^{lm}(P)\EP(h_t^k-\EP h_t^k)^2(\xitil_m-\EP\xitil_m) \partial_l \\
  & = & g^{lm}(P)\EP\left((h_t^k-\EP h_t^k)^2-\EP(h_t^k-\EP h_t^k)^2\right)
        (\xitil_m-\EP\xitil_m)\partial_l,
\end{eqnarray*}
and
\[
\sum_k\nabla_{\bfV_{k,t}}^{(-1)}\bfV_{k,t}(P)
  = 2g^{lm}(P)\left\langle\bfW_t(P) , \partial_m\right\rangle_P \partial_l,
\]
which shows that $\sum_k\nabla_{\bfV_{k,t}}^{(-1)}\bfV_{k,t}=2\bfW_t\in C^\infty(N;TN)$.
So (\ref{eq:fdfil}) has a strong solution, $\Psi$.  The fact that $\Pi=\Psi(P_0,\nu)$
follows from (\ref{eq:efil}), which is the $e$ representation of (\ref{eq:fdfil}).
\end{proof}

\section{Examples} \label{se:examples}

\subsection{An infinite-dimensional diffusion filter}  \label{se:infdex}

This example is developed from that in section 8.6.2 of \cite{lish1}.  The
signal is a special case of the diffusion process of Example \ref{ex:diffsig} in
section \ref{se:mvnlf}, in which $m=d=1$, $a\equiv 1$, and $b$ and $h$ satisfy
\begin{eqnarray}
& |b(x)|+|b'(x)|+|b''(x)|+|b'''(x)|+|h(x)|+|h'(x)|+|h''(x)|
  \le C, & \nonumber \\
& \hfill|b'''(x)-b'''(y)|+|h''(x)-h''(y)|
  \le C|x-y|, \hfill & \label{eq:infdbnd}
\end{eqnarray}
for all $x,y\in\R$ and some $C<\infty$, $P_0=N(0,R)$ (the Gaussian measure with mean
zero and variance $R>0$), and $\mu$ has density $2^{-1}\exp(-|x|)$ with respect to
Lebesgue measure.

\begin{proposition} \label{pr:infdex}
The diffusion process $(X,Y)$ defined above satisfies (F1--F6).
\end{proposition}

\begin{proof}
(F1) is easily verified and, since $h$ is bounded, (F2) is immediate.  According to
Lemma 8.5 in \cite{lish1}, $X_t$ admits the following $\cy_t$-conditional
density with respect to Lebesgue measure:
\[
\pitil_t(x)
  = \int\Eouttil\exp(B(x)-B(y)+\Gamma_t(x,y))n(y,t)(x)n(0,R)(y)dy,
\]
where $B(x):=\int_0^x b(y)dy$, $n(m,v)$ is the mean $m$, variance $v$ Gaussian density,
\[
\Gamma_t(x,y)
  := \int_0^t (h-\hhbar_s)(\Xtil_s^{y,t,x})\,d\nu_s
      - \half\int_0^t\left((h-\hhbar_s)^2+b^2+b'\right)(\Xtil_s^{y,t,x})\,ds,
\]
and $(\Xtil_s^{y,t,x}, s\in[0,t])$ is a Brownian motion on an auxiliary probability
space $(\Omegatil,\cftil,\PRtil)$, pinned to the values $y$ at $s=0$ and $x$ at $s=t$.
$\Xtil^{y,t,x}$ can be expressed in terms of a Brownian bridge process
$(\Wtil_s^t, s\in[0,t])$ as $\Xtil_s^{y,t,x}=sx/t + (t-s)y/t + \Wtil_s^t$. (See
Corollary 8.6 in \cite{lish1}.  NB.~Equations (8.97) and (8.108) in
\cite{lish1} contain some typographical errors, which are corrected in the
above.)  The $\cy_t$-conditional distribution of $X_t$ thus admits the strictly positive
density $\pi_t:=\pitil_t/r$ with respect to $\mu$, where $r=2^{-1}\exp(-|x|)$.

Theorem 8.7 in \cite{lish1} shows that $\pi$ satisfies (F6).  In particular,
$\pi$ is $(\cy_t)$-adapted for each $x$ and continuous in $(t,x)$, and hence
$\cp_Y\times\cx$-measurable. It thus satisfies (F3).  Equations (8.123) and (8.124) in
\cite{lish1} enable the explicit calculation of $\ca\pi_t$; straightforward
calculations (involving integration by parts) show that
\[
(\ca\pi_t)(x) = \frac{1}{2r}\int\Eouttil\gamma_t(x,y)
                  \exp(B(x)-B(y)+\Gamma_t(x,y))n(y,t)(x)n(0,R)(y)\,dy,
\]
where
\begin{eqnarray*}
\gamma_t(x,y)
  & := & -b^2(x) + \left(b(y)-\frac{\partial\Gamma_t}{\partial x}
         - \frac{\partial\Gamma_t}{\partial y} + \frac{y}{R}\right)^2 \\
  &    & \quad -b'(x)-b'(y)+\frac{\partial^2\Gamma_t}{\partial x^2}
         + 2\frac{\partial^2\Gamma_t}{\partial x\partial y}
         + \frac{\partial^2\Gamma_t}{\partial y^2} - \frac{1}{R}.
\end{eqnarray*}
(A proof of the existence and continuity of the derivatives of $\Gamma$ is contained
in the proof of Lemma 8.8 in \cite{lish1}.)  In particular, $\ca\pi$ is
$(\cy_t)$-adapted for each $x$ and continuous in $(t,x)$, and hence
$\cp_Y\times\cx$-measurable.  The derivatives of $\Gamma$ can be computed in closed
form;  for example
\[
\frac{\partial\Gamma_t}{\partial x}
  = \int_0^t \frac{s}{t}h'(\Xtil_s^{y,t,x})\, d\nu_s
    - \int_0^t \frac{s}{t}((h-\hhbar_s)h'+bb'+b''/2)(\Xtil_s^{y,t,x})\, ds.
\]

The joint density $n(y,t)(x)n(0,R)(y)$ can be written in the
$x$-marginal/$y$-conditional form, $n(\alpha_tx,\sigma_t^2)(y)n(0,R+t)(x)$, where
$\alpha_t:=R/(R+t)$ and $\sigma_t^2:=Rt/(R+t)$, and so
\begin{eqnarray}
\pi_t(x)
  & = & \frac{n(0,R+t)}{r}(x)\int\Eouttil\exp(B(x)-B(y)+\Gamma_t(x,y)) \nonumber \\
  &   & \qquad\qquad\qquad\qquad\qquad\qquad
        \times n(\alpha_tx,\sigma_t^2)(y)\,dy,\qquad \label{eq:mudens1} \\
(\ca\pi_t)(x)
  & = & \frac{n(0,R+t)}{2r}(x)\int\Eouttil\gamma_t(x,y)
        \exp(B(x)-B(y)+\Gamma_t(x,y)) \nonumber  \\
  &   & \qquad\qquad\qquad\qquad\qquad\qquad
        \times n(\alpha_tx,\sigma_t^2)(y)\,dy. \label{eq:mudens2}
\end{eqnarray}
Since $|b|,|b'|,|h|\le C$, for any $k\in\N$ and any $x,y\in\R$,
\begin{eqnarray}
\Eout\Eouttil\exp(k\Gamma_t(x,y))
  & \le & \Eout\Eouttil\Xi_t(x,y)\exp((2k(k-1)C^2+k(C^2+C)/2)t) \nonumber \\
  &  =  & \exp((2k(k-1)C^2+k(C^2+C)/2)t), \label{eq:egambnd}
\end{eqnarray}
where $\Xi(x,y)$ is the exponential martingale
\[
\Xi_t(x,y) := \exp\left(k\int_0^t (h-\hhbar_s)(\Xtil_s^{y,t,x})\,d\nu_s
              - \frac{k^2}{2}\int_0^t (h-\hhbar_s)^2(\Xtil_s^{y,t,x})\,ds\right).
\]
Furthermore, since $|B(y)|\le C|y|$,
\begin{eqnarray}
\int\exp(-kB(y))n(\alpha_tx,\sigma_t^2)(y)\,dy
  & \le & 2\int\cosh(kCy)n(\alpha_tx,\sigma_t^2)(y)\,dy \nonumber \\
  &  =  & 2\exp(k^2C^2\sigma_t^2/2)\cosh(kC\alpha_tx) \nonumber \\
  & \le & 2\cosh(kCx)\exp(k^2C^2t/2).  \label{eq:expmbnd}
\end{eqnarray}
Applying Jensen's inequality to (\ref{eq:mudens1}), we obtain
\[
\pi_t^2(x)
  \le \frac{n(0,R+t)^2}{r^2}(x)\int\Eouttil\exp(2(B(x)-B(y)+\Gamma_t(x,y)))
      n(\alpha_tx,\sigma_t^2)(y)\,dy,
\]
and so it follows from (\ref{eq:egambnd}) and (\ref{eq:expmbnd}) with $k=2$ that
\begin{eqnarray}
\Eout\Emu\pi_t^2
  & \le & 4\exp((7C^2+C)t)\int\cosh^2(2Cx)\frac{n(0,R+t)^2}{r}(x)\,dx \nonumber \\
  & \le & K_T < \infty \quad{\rm for\ all\ }t\in[0,T] {\rm\ and\ any\ }T<\infty.
          \label{eq:pisqbnd}
\end{eqnarray}
Together with the boundedness of $h$, this shows that (F5) is satisfied.  Applying
Jensen's inequality and the Cauchy-Schwartz inequality to (\ref{eq:mudens2}), we obtain
\begin{eqnarray*}
\Eout(\ca\pi_t)^2(x)
  & \le & \frac{n(0,R+t)^2}{4r^2}(x)
          \int\sqrt{\Eout\Eouttil\gamma_t(x,y)^4} \\
  &     & \quad\times\sqrt{\Eout\Eouttil\exp(4(B(x)-B(y)+\Gamma_t(x,y)))}
          n(\alpha_tx,\sigma_t^2)(y)\,dy.
\end{eqnarray*}
It follows from the bounds in (\ref{eq:infdbnd}), and standard properties of the Lebesgue
and It\^{o} integrals that, for any $x,y\in\R$,
\[
\Eout\Eouttil\gamma_t(x,y)^4 \le K(1+t^8)(1+y^8) \quad {\rm for\ some\ }K<\infty.
\]
Following the same steps as were used in the proof of (\ref{eq:pisqbnd}) (but using
$k=4$ in (\ref{eq:egambnd}) and (\ref{eq:expmbnd})) we now conclude that
\begin{equation}
\Eout\Emu(\ca\pi_t)^2
  \le \tilde{K}_T
   <  \infty \quad{\rm for\ all\ }t\in[0,T] {\rm\ and\ any\ }T<\infty. \label{eq:aspsqbnd}
\end{equation}
A further application of Jensen's inequality to (\ref{eq:mudens1}) yields
\[
\pi_t(x)^{-1}
  \le \frac{r}{n(0,R+t)}(x)\int\Eouttil\exp(-B(x)+B(y)-\Gamma_t(x,y))
      n(\alpha_tx,\sigma_t^2)(y)\,dy,
\]
which, together with (\ref{eq:mudens2}), shows that
\begin{eqnarray*}
\left|\pi_t(x)^{-1}\ca\pi_t(x)\right|
  & \le & \half\int\Eouttil|\gamma_t(x,y)|
          \exp(-B(y)+\Gamma_t(x,y))n(\alpha_tx,\sigma_t^2)(y)\,dy \\
  &     & \quad \times\int\Eouttil\exp(B(y)-\Gamma_t(x,y))n(\alpha_tx,\sigma_t^2)(y)\,dy,
\end{eqnarray*}
and the bound,
\begin{equation}
\Eout\Emu\left(\pi_t^{-1}\ca\pi_t\right)^2
  \le \exp(K(1+t)) \quad {\rm for\ some\ }K<\infty, \label{eq:pasbnd}
\end{equation}
easily follows.  The bound (\ref{eq:f41}), for any $T<\infty$, follows from
(\ref{eq:aspsqbnd}) and (\ref{eq:pasbnd}), and this establishes (F4).
\end{proof}

\subsection{A Kalman-Bucy Filter} \label{se:kbfilt}

This is another example in which the signal is a diffusion process.  Here $b(x)=Bx$,
$a(x)=A$, $h(x)=Cx$ and $P_0=N(m_0,R_0)$, where $B$ is an $m\times m$ matrix, $A$ is
a positive semi-definite $m\times m$ matrix, $C$ is a $d\times m$ matrix, and $R_0$ is
a positive definite $m\times m$ matrix.  The posterior distribution is
$\Pi_t=N(\Xbar_t,R_t)$, where the mean vector, $\Xbar_t$, and covariance matrix, $R_t$,
satisfy the Kalman-Bucy filtering equations \cite{lish1}:
\begin{eqnarray}
\Xbar_t  & = & m_0 + \int_0^t B\Xbar_s\,ds + \int_0^t R_sC^*\,d\nu_s \label{eq:kbfilt1} \\ 
R_t      & = & R_0 + \int_0^t (BR_s + R_sB^* + A - R_sC^*CR_s)\,ds. \label{eq:kbfilt2}
\end{eqnarray}
It is well known that such Gaussian measures belong to finite-dimensional exponential
statistical manifolds.  In order to apply the results of sections \ref{se:mvnlf} and
\ref{se:finexp}, we construct such a manifold as a $C^\infty$-embedded submanifold of
$M(\R^m,\mu)$, where $\mu$ has density $2^{-m}\exp(-\sum_j|x^j|)$ with respect to
Lebesgue measure.  Let $\SI$ be the set of symmetric positive definite $m\times m$
real matrices, and let $n:=m(m+3)/2$.  For any $y\in\R^n$, let $\alpha(y)$ be the
$m$-vector comprising the first $m$ elements of $y$, and let $\beta(y)$ be the
symmetric $m\times m$ matrix whose lower triangle contains the elements
$y^{m+1},y^{m+2},\ldots,y^n$ in some fixed arrangement.  Then
$G:=(\alpha,\beta)^{-1}(\R^m\times\SI)$ is an open subset of $\R^n$, and the map
$(\alpha,\beta):G\rightarrow \R^m\times\SI$ is a linear bijection.
Let $\etil:\bX\times\R^n\times\R\rightarrow\R$ be defined by
\[
\etil(x,y,z) := -\half x^*\beta(y)x + \alpha(y)^*x + z\sum_{j=1}^m|x^j|;
\]
then $\Lambda\etil(\fndot,y,z)\in e(M)$ for all $(y,z)\in G\times\R$.  Let
$\xi_i:=\Lambda\etil(\fndot,\bfe_i,0)$, where $(\bfe_i,\,1\le i\le n)$ is the coordinate
orthonormal basis in $\R^n$, let $\xi_{n+1}:=\Lambda\etil(\fndot,0,1)$, and let
$\Ntil:=e^{-1}\circ\gamma(G\times\R)$, where $\gamma(y,z) := y^i\xi_i+z\xi_{n+1}$.
$\Ntil$ is an $n+1$-dimensional instance of the exponential manifold discussed in
section \ref{se:finexp}. The $n$-dimensional submanifold
$N:=e^{-1}\circ\gamma(G\times\{1\})$ comprises all the non-singular Gaussian measures
on $\R^m$.

\begin{proposition} \label{pr:kbfilt}
The diffusion process $(X,Y)$ defined above satisfies (F1--F6), and (F7--F10) with
respect to the $n$-dimensional submanifold $N$.
\end{proposition}

\begin{proof}
(F7) (and hence (F1)) follows from the fact that $R_t\in\SI$ for all $t$; (F2)
and (F9) are obvious.  Straightforward calculations show that, for any $P\in N$,
\begin{eqnarray*}
\frac{\ca p}{p}(x)
  & = & \half x^*\beta(y)(A\beta(y)+2B)x - \alpha(y)^*(A\beta(y)+B)x \\
  &   & \quad +\half\big(\alpha(y)^*A\alpha(y)-\tr(A\beta(y)+2B)\big),
\end{eqnarray*}
where $y=\theta(P)$, and (F8) and (F10) readily follow.  (F3--F6) are easily verified
from (\ref{eq:kbfilt1},\ref{eq:kbfilt2}).
\end{proof}

\subsection{Wonham's Filter} \label{ex:wonham}

In this, $\bX$ and $\ca$ are as defined in Example \ref{ex:discsig} of section
\ref{se:mvnlf}, $X$ is a Markov jump process for which $\PR(X_0=x)>0$ for all $x$,
and $\mu$ is the uniform probability measure.  $M$ is itself an $n$ ($=m-1$)-dimensional
exponential statistical manifold.  In the set-up of section \ref{se:finexp},
appropriate choices are $G=\R^n$ and
$\xi_i := \indic_{\{i\}}-n^{-1}\sum_{j\neq i}\indic_{\{j\}}$.
(F1--F10) are easily verified.

\section{Concluding Remarks} \label{se:concl}

This paper developed information geometric representations for nonlinear filters in
continuous time, and studied their properties.  Information manifolds are natural state
spaces for the posterior distributions of Bayesian estimation problems where many
statistics are required.  They clarify information-theoretic properties of estimators,
and their metrics are appropriate ``multi-objective'' measures of approximation error.
The results also have bearing on the theory of non-equilibrium statistical mechanics,
in which rates of entropy production can be associated with rates of information supply
\cite{newt3}, and hence with the quadratic variation of a process of ``mesoscopic states''
in a particular pseudo-Riemannian metric.

The development of approximations is beyond the scope of this paper.  However, we
conclude with a few remarks on this issue.  One approach is to first define an
appropriate differential equation (evolution equation) to which numerical methods might
be applied.  Equations (\ref{eq:hfilt}) and (\ref{eq:fdfil}) are expressed in terms of
the innovations process $\nu$ in order to emphasise their information theoretic
properties. Substituting for $\nu$ in (\ref{eq:hfilt}), we obtain an $H$-valued It\^{o}
equation for the nonlinear filter in terms of the observation process $Y$:
\begin{equation}
\phi(\Pi_t) = \phi(P_0) + \int_0^t(u_s-z_s)\,ds + \int_0^t v_s\,dY_s, \label{eq:ydriven}
\end{equation}
where $z_t := \zeta_t + v_t\hhbar_t$.  If $h$ is bounded, then $z$ and $ve_k$ can be
expressed in terms of the time-dependent, locally Lipschitz vector fields
$\bfz,\bfv_k:[0,\infty)\times H\rightarrow H$, where
\begin{eqnarray}
\bfz_t(a)     & = & \Lambda\left(\half|h_t-\EP h_t|^2+(p+1)(h_t-\EP h_t)^*\EP h_t\right)
                    \\
\bfv_{k,t}(a) & = & \Lambda(p+1)(h_t^k-\EP h_t^k),
\end{eqnarray}
and $P=\phi^{-1}(a)$.  However, except in special cases such as the exponential filters
of section \ref{se:finexp}, $u$ is more problematic since the {\em infinitesimal}
characterisation of $P_t$ in (\ref{eq:kolfor}) is dependent on the topology of the
signal space $\bX$.  The topology of $M$ arises from purely measure-theoretic
constructs, and is not dependent on the existence of a topology on $\bX$.
(\cite{newt4} assumes only that $(\bX,\cx,\mu)$ is a probability space.)
This is quite natural in the context of Bayesian estimation and, in particular,
nonlinear filtering: Bayes' formula and Shannon's information quantities are
measure-theoretic in nature, as is the Markov property in its most general form (a
property of conditional independence).  It may be possible to overcome this problem by
strengthening the topology of $M$ in some way (for example, by the use of Sobolev space
techniques in the case  of filters for diffusion signals).  However, this is not
necessarily the best approach; for the purposes of approximation, it suffices to solve
a simpler evolution equation for an approximate filter.  This idea is developed in
\cite{bhlg1}, where approximations to $\Pi$ are constrained to remain on
finite-dimensional exponential statistical manifolds, on which projections of the
processes $u$, $z$ and $v_k$ can be represented in terms of locally Lipschitz
vector fields.  The manifold $M$ contains a rich variety of smoothly embedded
submanifolds, to which this method could be generalised \cite{newt4}.  The selection
of a good submanifold for a particular problem, together with a suitable coordinate
system, would be critical to this approach.

Hilbert-space-valued filtering equations based on the {\em Zakai equation} may be
more suitable for these purposes.  Manifolds of finite (un-normalised) measures, to
which the Zakai equation could be lifted, are developed in \cite{newt5}.
These avoid the normalisation constant $Z(a)$ in the computation of the density.

Another approach to the problem of approximation would be to switch to a discrete-time
model ``up front'', replacing $\Pi$ by a time sampled version.  This would replace the
It\^{o} equation (\ref{eq:ydriven}) by a difference equation, on which approximations
could be based.  This would avoid the Kolmogorov forward equation (\ref{eq:kolfor}),
replacing it by an integral equation (the Chapman-Kolmogorov equation) over each time
step, and thereby eliminating problems concerning the topology of $\bX$.  (In the case of
diffusion signal processes, for example, the transition measure over a short time step
could be approximated by an appropriate Gaussian.)

Time reversal is used in \cite{newt1,newt2,newt3} to construct {\em dual}
filtering problems, in which the primal signal and filter processes exchange roles.
In the notation of this article, the dual filter computes the process of posterior
distributions for the primal filter $\Pi$ (regarded as a dual signal) in reverse time,
based on a dual observation process.  Such posterior distributions take values in the
set of probability measures on $M$.  Since $M$ is itself a complete, separable metric
space, one can easily use the construction of section \ref{se:manif} to define a (dual)
Hilbert manifold of such probability measures.  However, a striking feature of the dual
filter is that it is parametrised by the primal signal process $X$, reversed in time.
In this way, the topology of the primal signal space $\bX$ is connected with the
information topology of the dual problem.  It may be possible to exploit this fact in
filter approximations.

Notions of information {\em supply} and {\em dissipation} for nonlinear filters are
defined in \cite{newt2,newt3}.  The supply at time $t$ is the mutual
information $I(X;Y_0^t)$, and the dissipation is the $X_t$-conditional variant,
$\Eout I(X;Y_0^t|X_t)$.  Modulo initial conditions, the supply of the primal filter
is the dissipation of its dual, and vice-versa \cite{newt3}.  The quantity
$\Eout I(X;Y_0^t|X_t)$ was studied in \cite{maza1} in the context of filters for
diffusion processes, and shown to be connected with the Fisher metric in the sense that
\[
\Eout I(X;Y_0^t\cond X_t)
  = \half\int_0^t\Eout\left(\nabla\log\frac{\pi_s}{p_s}\right)^*
    a\left(\nabla\log\frac{\pi_s}{p_s}\right)(X_s)\,ds,
\]
where $p$ and $\pi$ are the prior and posterior densities, and $a$ is the diffusion
matrix for the signal.  The integral here is the average quadratic variation of the
dual filter in the Fisher metric, and the integrand is the mean-square error for the
dual observation function $h^d(X_s,p):=(\sigma^*\nabla\log p)(X_s)$, where $\sigma$
is a matrix square-root of $a$ \cite{newt3}.

\end{document}